\documentclass[12pt]{article}
\usepackage{amssymb}
\usepackage{amsmath,amsfonts,amssymb}

\newcommand{\R}{{\mathbb R}}

\usepackage{bbm}

\def\F{{\cal F}}

\newtheorem{defn}{Definition}[section]
\newtheorem{theo}[defn]{Theorem}
\newtheorem{lem}[defn]{Lemma}
\newtheorem{prop}[defn]{Proposition}
\newtheorem{cor}[defn]{Corollary}
\newtheorem{rem}[defn]{Remark}

\newenvironment{proof}{{\bf Proof }}{{\vskip 0.1cm \hfill$\Box$}}

\begin{document}
\centerline{\Huge \sf Pathwise uniqueness of the squared Bessel }
\vspace{0.5cm}
\centerline{\Huge \sf and CIR processes with skew reflection}
\vspace{0.5cm}
\centerline{\Huge \sf  on a deterministic time dependent curve}
\vspace{1cm}
\centerline{{\large \sf Gerald  TRUTNAU}{\footnote{Supported by the Research Settlement Fund for new faculty of Seoul National University, the research project "Advanced Research and Education of Financial Mathematics" at Seoul National University, the SFB-701 at Bielefeld University, and the BIBOS-Research Center at Bielefeld University.}}}
\vspace{1cm}
\hspace{-0.65cm}Department of Mathematical Sciences and Research Institute of Mathematics,
Seoul National University,
San56-1 Shinrim-dong, Kwanak-gu,
Seoul 151-747, South Korea
 (e-mail: trutnau@snu.ac.kr, tel.: +82 (0)2 880 2629; fax: +82 (0)2 887 4694.)\\ \\
{\bf Abstract:} Let $\sigma,\delta>0, b\ge0$.  Let $\lambda^2:\R^+\to \R^+$, be continuous, and locally 
of bounded variation. We develop a general analytic criterion for the pathwise uniqueness of 
\begin{eqnarray*}
R_t=R_0+\int_0^t\sigma\sqrt{|R_s|}dW_s+\int_0^t\frac{\sigma^2}{4}(\delta-bR_s)ds+ (2p-1)\ell^0_t(R-\lambda^2),
\end{eqnarray*} 
where $p\in (0,1)$, and $\ell^0_t(R-\lambda^2)$ is the symmetric semimartingale 
local time of $R-\lambda^2$. The criterion is related 
to the existence of nice (Kummer) functions for the time dependent 
infinitesimal generator of $R$. As a corollary we obtain explicit sufficient conditions for pathwise uniqueness. 
These are expressed in terms of $\lambda^2$, its derivative, and the parameters $\sigma,\delta,b,p$. 
\section{Introduction}
In this article we investigate pathwise uniqueness of squared Bessel and Cox-Ingersoll-Ross
(CIR)  processes that are perturbed by a constant multiple of a local time on a deterministic time dependent curve. These processes 
form a natural generalization of the classical squared Bessel and CIR processes. Local times on curves already appear in \cite{dav}, \cite{ei}, \cite{bb}, \cite{bcs}, in \cite{peskir} 
where an It\^o formula involving such local times is derived, 
and in \cite{RuTr4}, \cite{Tr5}. 
For another related It\^o-type formula but which is less relevant for the topic of our paper we refer to  \cite{etz}.
The tractability of  stochastic differential equations involving local times on curves (resp. on surfaces) has been facilitated by the work \cite{peskir} (resp. \cite{peskir3}) with resulting recent applications in  financial mathematics, see e.g. \cite{mij}, \cite{peskir2}. Quite different natural generalizations of squared Bessel processes without occuring local time such as squared Bessel processes and squared
Ornstein-Uhlenbeck processes with negative dimensions or negative starting points were studied in \cite{GoeYor}.
\subsection{Motivations}
Our motivations are various. 
First, we can regard the perturbation as introduction of a new parameter that provides potentially additional insight on (the singularity and the properties of) the classical original processes. 
This is indicated further below in \ref{exandunique} where we explain our pathwise uniqueness results. 
From a theoretical point of view we are interested in general pathwise uniqueness results for diffusions in dimension one whose diffusion and drift coefficients 
satisfy the Yamada-Watanabe conditions (see the original article \cite{YamWat}, or e.g. the textbook \cite[Chapter IX, \S 3]{RYor}) and which have an  additional concrete reflection term.
The here considered skew reflected processes are special but typical examples of such diffusions with non-Lipschitz diffusion coefficient and parabolic local time. 
In particular, our developed method for proving pathwise uniqueness gives a first hint how to treat the challenging general problem. 
To our knowledge it is the first time that such singular equations (degenerate non-Lipschitz diffusion coefficient plus local time on a curve), which are in close relation to boundary crossing problems and hitting times of 
deterministic curves of the original processes, are studied. The boundary crossing problem itself is classical, see e.g. \cite{khi}, \cite{ler} and references therein, \cite{dur}, and in particular \cite{GoeYor} and references therein for the case of Bessel processes where the given curve reduces to a constant line, i.e. a so-called barrier. Quite important and intensive work on pathwise uniqueness of one dimensional stochastic differential equations with non-trivial coefficients and reflection term has been done (see e.g. \cite{veret}, \cite{LeGall2}, \cite{bp}, \cite{rut}, \cite{ouk2}, \cite{mak}). However, it is remarkable that even in the simplest case when the time dependent curve reduces to a constant, 
none of these covers our considered processes with the\lq\lq standard\rq\rq coefficients in dimension one. Thus, there definitely arises a theoretical need for the investigation of such equations.
Our final motivations concern applications. Processes of the studied type and their square roots can be well applied in  financial mathematics (see e.g. \cite{dgs}, \cite{dgs2}, and below). Therefore, the tractability of such processes is an important issue.
\subsection{The SDE: detailed description and explanations}
\subsubsection{The SDE and its special cases}
We fix a continuous function  $\lambda^2:\R^+\to \R^+$ that is locally 
of bounded variation. For parameters $\sigma,\delta>0, b\ge 0$, $p\in(0,1)$, consider a solution to
\begin{eqnarray}\label{skewbesq}
R_t=R_0+\int_0^t\sigma\sqrt{|R_s|}dW_s+\int_0^t\frac{\sigma^2}{4}(\delta-bR_s)ds+ (2p-1)\ell^0_t(R-\lambda^2).
\end{eqnarray} 
where $\ell^0_t(R-\lambda^2)$ is the {\it  symmetric semimartingale 
local time} of $R-\lambda^2$ (for a reason to consider the symmetric semimartingale local time $\ell^0(R-\lambda^2)$ see Remark \ref{first}(ii)).
Due to the presence of the square root in the diffusion part, the local time in equation (\ref{skewbesq}) vanishes if $\lambda\equiv 0$. In fact, this is a well-known direct consequence of the
occupation time formula (cf. e.g. proof of Lemma \ref{timezero}(ii)).
Hence for $\lambda\equiv 0$, or $p=\frac12$ we obtain the CIR process. Starting from \cite{cir} the CIR process 
has been used intensively to model the evolution of the interest rate in a financial market (see e.g. again \cite{GoeYor} and references therein). For 
$\sigma=2$, $b=0$, and $\lambda^2\equiv 0$, or $p=\frac12$  we obtain the squared Bessel process of dimension $\delta$. 
Recently, in \cite{dgs}, the square root of (\ref{skewbesq}) for $\sigma=2, b= 0$, $p\in(0,1)$, dimension $\delta\ge 2$, and $\lambda\equiv a\in\R^{++}$, was independently of us constructed and further studied. 
In an earlier paper  \cite{RuTr4} we considered a more general equation where $\delta >0$ and $\lambda$ is a monotone function, and later in  \cite{Tr5} we generalized our results and considered weakly differentiable $\lambda$.
In \cite{dgs} many interesting properties are derived and applied to the valuation of perpetuities and pricing of weighted Asian options. 
The authors of \cite{dgs} call these processes {\it asymmetric skew Bessel processes} in order to distinguish them from the {\it skew Bessel processes} defined in \cite{bpy}. 
\subsubsection{Positivity of the SDE}
As in the classical case $p=\frac12$ a solution to (\ref{skewbesq}) 
always stays positive when started with positive initial condition.  
We show this in Lemma \ref{timezero}(ii) with the help of Tanaka's formula. 
Note that in equation (\ref{skewbesq}) we may have $2p-1<0$ and then comparison results such as e.g. the one in \cite{Yan1} only show that a solution to (\ref{skewbesq}) is smaller, hence more singular, 
than the classical CIR process. Thanks to the derived positivity one can as in the non-perturbed classical case a posteriori discard the absolute value under the square root in (\ref{skewbesq}). 
\subsubsection{Skew reflection}
The reflection term in (\ref{skewbesq}) has a prefactor $2p-1$. Diffusions involving such a term 
with $p\not= 0,\frac12,$ or $1$ are called skew reflected diffusions. If $p=0$, or $p=1$ we speak about reflected diffusions. In contrast to reflected diffusions, skew reflected diffusions 
behave more like ordinary diffusions (i.e. $p=\frac12$) in the sense that the nature of the state space is not influenced. Reflected diffusions split the state space in two separate parts from which one can not be left if it is once entered. Skew reflected diffusions were studied by many authors (see \cite{lejay} and 
references therein) starting from \cite{itomckean}, and \cite{walsh}, and
provide typical examples of diffusions with discontinuous local times (see Lemma \ref{discont}, and \cite{walsh}).
For nice coefficients the skew reflection has been described rigorously in terms of excursion theory, such 
as e.g. for the skew Brownian motion (cf. \cite{walsh}, \cite{hs}, \cite{lejay}). 
Heuristically, one has a reflection downwards 
with \lq\lq tendency\rq\rq\ $1-p$ and upwards with  \lq\lq tendency\rq\rq\ $p$ where $2p-1=p-(1-p)$. 
For various applications of skew Brownian motion and other skew reflected 
processes see \cite{lejay} and references therein. Besides it was noted in \cite{Wein2} that stochastic differential equations with prefactored local time may be used to 
model diffusion processes in a medium with permeable barrier. For more on skew reflection and permeable barriers we also refer to \cite{port} and \cite{aryport2}.
\subsubsection{The singularity of the SDE}
In addition to the singular diffusion coefficient, the specialty of  (\ref{skewbesq}) is that the reflection takes place whenever $R$ meets the  given time dependent curve $\lambda^2$. 
Moreover, equation (\ref{skewbesq}) can be transformed (see \cite[p. 382]{Tr5}) into a degenerate, non-Lipschitz, and time-dependent equation with reflection term in dimension one. 
For this type of equations general uniqueness and weak existence results are unknown and could possibly even not hold. 
\subsection{Existence and uniqueness results}\label{exandunique}
Weak existence of a solution to  (\ref{skewbesq}) as well as for its square root has been established in \cite{Tr5} in various cases.  
It is worth to remark that we were not able to construct in whole generality a solution in the extreme cases $p=0$ and $p=1$, and for other than increasing (though nonetheless for constant) 
$\lambda^2$ if $-1<2p-1<0$. 
However, if $-1<2p-1<0$ pathwise uniqueness can hold even if $\lambda^2$ strictly decreases on some interval. 
If $|2p-1|>1$, then at least in case $\lambda^2$ is absolutely continuous there is no solution to (\ref{skewbesq}) (cf. Remark \ref{nosolution+}(ii)).\\
In this note we develop an analytic criterion for pathwise uniqueness of  
(\ref{skewbesq}). This is done in Theorem \ref{pathunique2}. 
Uniqueness is 
reduced to the resolution of a parabolic differential equation corresponding 
to the infinitesimal generator of $R$ (see (\ref{out}), and also (\ref{out2})).
The general criterion of Theorem \ref{pathunique2} is directly 
applied in Corollary \ref{pathunique0} in order to show that
pathwise uniqueness holds for 
(\ref{skewbesq}), whenever 
$$
d\lambda^2(s)\le \frac{\sigma^2 \delta}{4}ds, \ \ \mbox{ if } \ p\in \left (\frac{1}{2},1\right ),
$$
or, whenever
$$
d\lambda^2(s)\ge \frac{\sigma^2}{4}\left (\delta-b \lambda^2(s)\right )ds,\ \ \mbox{ if }\ p\in \left (0,\frac{1}{2}\right ).
$$
The inequalities are to be understood in the sense of signed measures on $\R^+$. For instance, if $0<2p-1<1$ (i.e. $\frac12< p<1 $), $\sigma=2$, this 
means that the increasing part of 
$\lambda^2$ is Lipschitz continuous with Bessel dimension $\delta$ as Lipschitz constant, and that 
the decreasing  part is arbitrary. If $-1<2p-1<0$ (i.e. $0<p<\frac12$), and $\lambda^2$ is absolutely continuous we obtain 
\begin{eqnarray}\label{2}
(\lambda^2)'\ge \frac{\sigma^2}{4}(\delta-b\lambda^2)
 \end{eqnarray}
as a sufficient condition for pathwise uniqueness.
According to (\ref{2}), if $\lambda^2$ decreases on some interval it must be  above the mean-reverting level $\frac{\delta}{b}$ of the CIR process on that interval. 
Both results in case of $\frac12<p<1$, as well as in case $0<p<\frac12$, 
appear to be plausible to us in the following sense: in the first case, if $\lambda^2(t)$ grows faster than the largest possible drift $\frac{\sigma^2 \delta}{4}t$ of $R_t$, 
then there could 
be no difference between the reflected and the non-reflected equation, hence more than one solution; and if  $0<p<\frac12$ we argue that we can only add a 
negative singular drift to the classical process and keep pathwise uniqueness when the corresponding classical process 
also has a negative component of drift. The squared Bessel process only has a positive component of drift for all dimensions, but the CIR process also has a negative 
component of drift when it is above the mean-reverting level. Therefore, one might wonder whether these pathwise uniqueness results can be substantially improved or not. 
We leave this as an open question. However, in section \ref{three} we provide some tools to develop additional existence and uniqueness results. So one could start from there. 
In order to obtain our pathwise uniqueness result we made use of nice Kummer functions of the first kind (see Corollary \ref{pathunique0}). 
Even after \rq\rq localizing\lq\lq\ the main argument, we were not able to get any uniqueness result by using Kummer functions of the second kind (see however the proof of Corollary 
\ref{nosolution}(ii)). Kummer functions of the second kind in general do not behave well at the origin.
\subsubsection{The method for proving pathwise uniqueness}
Looking at the difference of $|R^{(1)}_t-R^{(2)}_t|$, 
where $R^{(1)},R^{(2)},$ are two solutions, we can not use {\it Le Gall's local time method} (see \cite{LeGall}),  
since although $\ell^0_t(R^{(1)}-R^{(2)})\equiv 0$, there always remains a term involving the 
local time on $\lambda^2$. The coefficients, as well as the parabolic situation, 
make transformations through harmonic functions as e.g. used in \cite{hs}, \cite{LeGall2}, \cite{rut}, impossible. 
In other words there seems to be no nice transformation of (\ref{skewbesq}) in an equation for which  pathwise uniqueness would be known. 
Therefore we have to find another method. Our line of arguments, is to first show that together with $R^{(1)},R^{(2)},$ the supremum 
$S=R^{(1)}\vee R^{(2)}$, and the infimum $I=R^{(1)}\wedge R^{(2)}$,  
is also a solution. This is a standard procedure. We refer to Remark \ref{first}(i) for a consequence of it. 
Then we have to find a nice function $H(t,x)$ (see Remark \ref{procedure}), 
stricly increasing in $x$, 
such that $\overline{g}(x-\lambda^2(t))^{-1}H(t,x)$ solves (\ref{out}), and to apply 
a generalized Gronwall inequality to the expectation of $H(t,S_t)-H(t,I_t)$ 
in order to conclude (see proof of Theorem \ref{pathunique2}, and Corollary \ref{pathunique0}). This is our contribution and it appears at least to us to be an efficient   
method for skew reflected equations with singular diffusion coefficient. 
Since sub/superharmonic functions w.r.t. the time homogeneous infinitesimal generator of the CIR process may lose of their advantageous 
properties under parabolic boundary conditions we cannot expect a positive answer for every curve, even not for constant ones.  
We emphasize that our technique can also be used for (\ref{skewbesq}) with more general drift coefficients. 
However, we feel that the details should be investigated elsewhere.\\   
In order to find that $S=R^{(1)}\vee R^{(2)}$, $I=R^{(1)}\wedge R^{(2)}$, are also solutions 
we use formulas on the representation of local times of the supremum and infimum of two semimartingales (see \cite{Wein}, \cite{oukrut}).
In order to make disappear the 
local time on $\lambda^2$ with the help of a nice function $H$, we use special  
It\^o-Tanaka formulas (see Lemma \ref{itotanaka}), which are proved using representation formulas for local times from \cite{oukrut} (see Lemma \ref{replocal2}). 
The derivation of these formulas takes in particular advantage of the fact that the time dependency is put into a  
semimartingale structure. Lemma \ref{itotanaka} is simple but useful, and allows $\lambda^2$ just to be locally of bounded variation.
It is also used in Corollary \ref{nosolution}(ii).
\subsubsection{Martingale problem and an additional pathwise uniqueness criterion}
In the third section we solve the martingale problem related to $R$ on a 
nice class of test functions (see Proposition \ref{mart}, and Remark \ref{alg} for its usefulness). 
We also add another pathwise uniqueness criterion in Theorem \ref{pathunique} which 
uses \lq\lq true\rq\rq\  time dependent functions and which is developed with the help of a recent extension of It\^o's formula from \cite{peskir}. \\

\section{Pathwise uniqueness in the non absolutely continuous case}\label{two}
Throughout this article $\mathbb{I}_A$ will denote the indicator function of a set $A$. 
We let $\R^+:=\{x\in \R|\,x\ge 0\}$. 
An element of $\R^+\times\R$ 
is typically represented as $(t,x)$, i.e. the first entry is always 
for time, the second always for space. The time derivative is denoted by $\partial_t$, 
the space 
derivative by $\partial_x$, and the  second space 
derivative by $\partial_{xx}$. Functions depending on space and time are 
denoted with capital letters, functions depending only on one variable are denoted with 
small case letters. If a function $f$ only depends on one variable we write $f'$, resp. $f''$, 
for its derivative, resp. second derivative.\\ \\
Let $\sigma,\delta>0$, $b\ge 0$, and $\lambda^2:\R^+\to\R^+$  
be continuous and locally of bounded variation. 
On an arbitrary complete filtered probability space 
$(\Omega, \F, (\F_t)_{t\ge 0}, P)$, consider an 
adapted continuous process with the following properties: 
$R$ solves the integral equation (\ref{skewbesq}) $P$-a.s, 
where
\begin{itemize}
\item[(i)] $(W_t)_{t\ge 0}$ is a $\F_{t}$-Brownian motion starting from zero,
\item[(ii)] $P[\int_0^t\{\sigma^2|R_s|+|\frac{\sigma^2}{4}(\delta-bR_s)|\}ds<\infty]=1$,
\item[(iii)] $\ell^0(R-\lambda^2)$ is the symmetric semimartingale local time of $R-\lambda^2$, i.e.
\begin{eqnarray}\label{symtanaka}
\frac12\ell^0_t(R-\lambda^2) & = & (R_t-\lambda^2(t))^+-(R_0-\lambda^2(0))^+\nonumber\\ 
&&-\int_0^t\frac{\mathbb{I}_{\{R_s-\lambda^2(s)>0\}}+
\mathbb{I}_{\{R_s-\lambda^2(s)\ge 0\}}}{2}d\{R_s-\lambda^2(s)\},\ \ t\ge 0.
\end{eqnarray} 
\end{itemize}
A process $R$ with the given properties is called a {\it (weak) solution} 
to (\ref{skewbesq}). In particular, one can show exactly as in \cite[\mbox{VI. (1.3) Proposition}]{RYor} 
that 
\begin{eqnarray}\label{support}
\int_0^t H(s,R_s) d\ell^{0}_s(R-\lambda^2) = \int_0^t H(s,\lambda^2(s))d\ell^{0}_s(R-\lambda^2),
\end{eqnarray}
for any positive Borel function $H$ on $\R^+\times \R$.\\ \\
We say that {\it pathwise uniqueness} holds for (\ref{skewbesq}), if, any two solutions 
$R^{(1)},R^{(2)}$, on the same filtered probability space $(\Omega, \F, P)$, with $R^{(1)}_0=R^{(2)}_0$ $P$-a.s., and with same Brownian motion, are $P$-indistinguishable, i.e. $P[R^{(1)}_t=R^{(2)}_t]=1$ for all $t\ge 0$.\\ \\
For later purposes we introduce the upper (or right) local time of $R-\lambda^2$
\begin{eqnarray}\label{tanaka}
\ell^{0+}_t(R-\lambda^2)=(R_t-\lambda^2(t))^+-(R_0-\lambda^2(0))^+-
\int_0^t\mathbb{I}_{\{R_s-\lambda^2(s)>0\}}d\{R_s-\lambda^2(s)\},
\end{eqnarray}
and the lower (or left) local time $\ell^{0-}(R-\lambda^2)$, which can be extracted from the following formula for the symmetric local time
\begin{eqnarray}\label{sym}
\ell^{0}(R-\lambda^2)=\frac{\ell^{0^+}(R-\lambda^2)+\ell^{0^-}(R-\lambda^2)}{2}.
\end{eqnarray} \\
Accordingly, $\ell^{0}(X), \ell^{0+}(X), \ell^{0-}(X)$, are defined for any continuous semimartingale $X$. 
Another useful formula, is the {\it occupation times formula}: If $X$ is a continuous semimartingale, then
\begin{eqnarray}\label{occtim}
\int_0^t H(s,X_s)d\langle X,X \rangle_s=\int_{\R}\int_0^t H(s,a)d\ell^{a+}_s(X)da \\ \nonumber
\end{eqnarray}
holds a.s. for every positive Borel function $H$ on $\R^+\times \R$, see e.g. \cite[\mbox{IV. (45.4)}]{rw}. 
By \cite[\mbox{VI. (1.7) Theorem}]{RYor} we may and will assume that the family of local times $(\ell^{a+}_t(X))_{(t,a)\in\R^+\times \R}$ 
is right-continuous, 
i.e. $(t,a)\mapsto \ell_t^{a+}(X)$ is a.s. continuous in $t$ and c\`adl\`ag in $a$.
Since $\ell^{a+}(X)$ has only countably many jumps in $a$, the formula holds for $\ell^{a}(X)$, and $\ell^{a-}(X)$, 
as well. \\ \\
The statements of the following lemma are direct consequences of well-known formulas. 
\begin{lem}\label{timezero} 
Let $R$ be a weak solution 
to (\ref{skewbesq}). Then:
\begin{itemize}
\item[(i)] $\int_0^t\mathbb{I}_{\{\lambda^2(s)= 0\}}d\ell^{0}_s(R-\lambda^2)=0$ $P$-a.s. for any $t\ge 0$. 
In particular 
$$
supp\{d\ell^{0}_s(R-\lambda^2)\}\subset \overline{\{\lambda^2(s)>0\}}.
$$
\item[(ii)] If $R_0\ge 0$ $P$-a.s., then $R_t\ge 0$  $P$-a.s. for any $t\ge 0$.\\
\item[(iii)] The time of $R$ spent at zero has Lebesgue measure zero, 
i.e. 
$$
\int_0^t\mathbb{I}_{\{R_s= 0\}}ds=0\ \ P\mbox{-a.s.} \ \ \forall t\ge 0. 
$$
\item[(iv)] The time of $R$ spent on $\lambda^2$ has Lebesgue measure zero, i.e. 
$$
\int_0^t\mathbb{I}_{\{R_s= \lambda^2(s)\}}ds=0\ \ P\mbox{-a.s.} \ \ \forall t\ge 0. 
$$
\end{itemize}
\end{lem}
\begin{proof}
(i) By (\ref{occtim}) we have
\begin{eqnarray*}
\int_{\R}\frac{\mathbb{I}_{\{a\not=0\}}}{|a|}
\int_0^t\mathbb{I}_{\{\lambda^2(s)=0\}}d\ell^{a+}_s(R-\lambda^2)da 
&=&\int_0^t\frac{\mathbb{I}_{\{R_s-\lambda^2(s)\not=0\}}}{|R_s-\lambda^2(s)|}
\mathbb{I}_{\{\lambda^2(s)=0\}}\sigma^2|R_s|ds\le \sigma^2 t.\\
\end{eqnarray*}
Since $\frac{1}{|a|}$ is not integrable in any neighborhood of zero, we obtain that 
$$
\int_0^t\mathbb{I}_{\{\lambda^2(s)=0\}}d\ell^{0+}_s(R-\lambda^2)=
\int_0^t\mathbb{I}_{\{\lambda^2(s)=0\}}d\ell^{0-}_s(R-\lambda^2)=0.
$$
The statement
thus holds for $\ell^{0+}(R-\lambda^2)$, and $\ell^{0-}(R-\lambda^2)$, and therefore 
also for $\ell^{0}(R-\lambda^2)$. \\
(ii) As a direct consequence of the occupation time formula 
$\ell_t^{0+}(R)\equiv 0$ (replace $\lambda^2$ by zero in the proof of (i)). 
Then, applying Tanaka's formula 
(cf. e.g. \cite[\mbox{VI. (1.2) Theorem}]{RYor}), using (i) and (\ref{support}),  
taking expectations, and stopping with $\tau_n:=\inf\{t\ge 0||R_t|\ge n\}$, we obtain
\begin{eqnarray*}
E[R_{t\wedge \tau_n}^-] &=&E[R_0^-] - E[\int_0^{t\wedge \tau_n}\mathbb{I}_{\{R_s\le 0\}}\frac{\sigma^2}{4}(\delta-bR_s)ds]\\
&& -(2p-1)E[\int_0^{t\wedge \tau_n}\mathbb{I}_{\{R_s\le 0\}}\mathbb{I}_{\{\lambda^2(s)>0\}}
d\ell^0_s(R-\lambda^2)]\\
&\le &-E[\int_0^{t\wedge \tau_n}\mathbb{I}_{\{R_s\le 0\}}\frac{\sigma^2}{4}(\delta-bR_s )ds]\le 0.
\end{eqnarray*}
It follows that $R_{t\wedge \tau_n}$ is $P$-a.s. equal to its positive part $R_{t\wedge \tau_n}^+$. Letting $n\to \infty$ concludes the proof.\\
(iii) Due to the presence of the square root in the diffusion part, we have 
$\ell_t^{0+}(R),\ell_t^{0-}(R)\equiv 0$ (replace $\lambda^2$ by zero in the proof of (i)). 
Using \cite[\mbox{VI. (1.7) Theorem}]{RYor}, (i) and (\ref{support}), 
it follows  $P$-a.s.
\begin{eqnarray*}
0\ =\ \ell_t^{0+}(R)-\ell_t^{0-}(R)
& = & \int_0^t\mathbb{I}_{\{R_s= 0\}}\left\{\frac{\sigma^2}{4}(\delta-bR_s)ds+
(2p-1)d\ell_s(R-\lambda^2))\right \}\\
& = & \frac{\sigma^2 \delta}{4}\int_0^t\mathbb{I}_{\{R_s= 0\}}ds.
\end{eqnarray*}
(iv) As a simple consequence of the occupation time formula, we have
\begin{eqnarray*}
\int_0^t\mathbb{I}_{\{R_s= \lambda^2(s)\}}\mathbb{I}_{\{ R_s\not=0\}}\sigma^2 |R_s|ds & = & 
\int_0^t \mathbb{I}_{\{R_s-\lambda^2(s)= 0\}}d\langle R-\lambda^2\rangle_s\\
&= & \int_{\R}\mathbb{I}_{\{0\}}(a)\ell_t^a(R-\lambda^2)da =0.\\
\end{eqnarray*}
But $P$-a.s. $\sigma^2|R_s|\mathbb{I}_{\{R_s\not= 0\}}>0$ $ds$-a.e. by (iii) and the assertion follows.
\end{proof}
\centerline{}
From the next lemma one observes at least when $\lambda^2$ is absolutely continuous the discontinuity of the 
local times in the space variable at zero.\\
\begin{lem}\label{discont}
Let $R$ be a weak solution 
to (\ref{skewbesq}). 
We have $P$-a.s.:
\begin{eqnarray*}
\ell^{0+}_t(R-\lambda^2)=2p\ell^{0}_t(R-\lambda^2)-
\int_0^t\mathbb{I}_{\{R_s= \lambda^2(s)\}}d\lambda^2(s),
\end{eqnarray*}
and
\begin{eqnarray*}
\ell^{0-}_t(R-\lambda^2)=2(1-p) \ell^{0}_t(R-\lambda^2)+
\int_0^t\mathbb{I}_{\{R_s= \lambda^2(s)\}}d\lambda^2(s).
\end{eqnarray*}
If $\lambda^2$ is absolutely continuous, i.e. $\lambda^2\in H^{1,1}_{loc}(\R^+)$, with  
$d\lambda^2(s)=(\lambda^2)'(s)ds$, then $P$-a.s.
\begin{eqnarray*}
\ell^0_t(R-\lambda^2)=\frac{1}{2p}\ell^{0+}_t(R-\lambda^2)=\frac{1}{2(1-p)} \ell^{0-}_t(R-\lambda^2).
\end{eqnarray*}
\end{lem}
\begin{proof}
Since $R-\lambda^2$ is a continuous semimartingale w.r.t. $P$, by Tanaka's formula 
(\ref{tanaka}) it follows $P$-a.s.
\begin{eqnarray*}
(R_t-\lambda^2(t))^+&=&(R_0-\lambda^2(0))^++\int_0^t\mathbb{I}_{\{R_s- \lambda^2(s)>0\}}d(R_s-\lambda^2(s))
+\frac12\ell^{0+}_t(R-\lambda^2).
\end{eqnarray*}
On the other hand, the symmetrized Tanaka formula (\ref{symtanaka}) 
together with Lemma \ref{timezero}(iii) gives
\begin{eqnarray*}
(R_t-\lambda^2(t))^+&=&(R_0-\lambda^2(0))^+
+\int_0^t\frac{\mathbb{I}_{\{R_s- \lambda^2(s)>0\}}+\mathbb{I}_{\{R_s- \lambda^2(s)\ge 0\}}}{2}
d(R_s-\lambda^2(s))\\
&&+\frac12\ell^{0}_t(R-\lambda^2)\\
&=&(R_0-\lambda^2(0))^+
+\int_0^t\mathbb{I}_{\{R_s- \lambda^2(s)>0\}}d(R_s-\lambda^2(s))-\frac12
\int_0^t\mathbb{I}_{\{R_s= \lambda^2(s)\}}d\lambda^2(s)\\
&&+p\ell^{0}_t(R-\lambda^2).
\end{eqnarray*}
Comparing the two formulas for $(R_t-\lambda^2(t))^+$ we obtain the first statement. The second follows 
from (\ref{sym}) by simple algebraic transformations. If $\lambda^2$ is absolutely continuous, then 
$$
\int_0^t\mathbb{I}_{\{R_s= \lambda^2(s)\}}d\lambda^2(s)=
\int_0^t\mathbb{I}_{\{R_s= \lambda^2(s)\}}(\lambda^2)'(s)ds=0
$$ 
by Lemma \ref{timezero}(iv), and the last statement follows.
\end{proof}\\
\begin{rem}\label{nosolution+}
(i) Using the previous Lemma \ref{discont} and (\ref{sym}), one can easily derive that 
$$
(2p-1)\ell^0_t(R-\lambda^2)=\frac{\ell^{0+}_t(R-\lambda^2)-\ell^{0-}_t(R-\lambda^2)}{2}+
\int_0^t\mathbb{I}_{\{R_s= \lambda^2(s)\}}d\lambda^2(s)
$$
and 
$$
(1-p)\ell^{0+}_t(R-\lambda^2)=p\ell^{0-}_t(R-\lambda^2)-
\int_0^t\mathbb{I}_{\{R_s= \lambda^2(s)\}}d\lambda^2(s).
$$
(ii) Let $\lambda^2$ be absolutely continuous. Then 
$\int_0^t\mathbb{I}_{\{R_s= \lambda^2(s)\}}d\lambda^2(s)=0$ by Lemma \ref{timezero}(iv). 
If now $|2p-1|>1$, then a solution to 
(\ref{skewbesq}) does not exist. This improves results of \cite[\mbox{Remark 2.7(i)}]{Tr5} where this could only be shown for 
those $\lambda^2$'s for which a solution to (\ref{skewbesq}) could be constructed. 
In fact, by (i) it holds that 
$\ell^0(R-\lambda^2),\ell^{0+}(R-\lambda^2),\ell^{0-}(R-\lambda^2)\equiv 0$. Now, in order to conclude see \cite[\mbox{Remark 2.7(i)}]{Tr5}. 
There analytic methods were used. 
In Corollary \ref{nosolution} below we provide another direct probabilistic proof for special $\lambda^2$'s and special initial conditions. 
\end{rem}
The next lemma is useful and appears in case of upper local times in \cite[\mbox{Lemme, p.74}]{Wein}. 
In the case of other local times (lower, and symmetric) it can easily directly be derived from \cite[\mbox{Lemme, p.74}]{Wein}. 
Using different and for themselves important 
formulas for the computation of local times Lemma \ref{replocaltime} is directly shown in \cite[\mbox{Corollary 2.6, and following remark}]{oukrut}. 
\begin{lem}\label{replocaltime} Let $X,Y$ be two continuous semimartingales, with $X_0=Y_0$. Suppose that $\ell^{0+}(X-Y)\equiv 0$.
Then the following representation formula holds for $\ell_s=\ell^{0+}_s$:
\begin{eqnarray}\label{replocaltime1}
\ell_t(X\vee Y)=
\int_0^t\mathbb{I}_{\{Y_s<0\}} d\ell_s(X)+
\int_0^t\mathbb{I}_{\{X_s\le 0\}}d\ell_s(Y).
\end{eqnarray}
Suppose that additionally $\ell^{0+}(Y-X)\equiv 0$. Then (\ref{replocaltime1}) holds also for 
$\ell_s=\ell^{0-}_s$, and $\ell_s=\ell^0_s$. In particular 
\begin{eqnarray}\label{replocaltime2}
\int_0^t\mathbb{I}_{\{Y_s=0\}} d\ell_s(X)=\int_0^t\mathbb{I}_{\{X_s=0\}}d\ell_s(Y)
\end{eqnarray}
holds for 
$\ell_s=\ell^{0+}_s$, $\ell_s=\ell^{0-}_s$, and $\ell_s=\ell^0_s$.\\
\end{lem}

\begin{lem}\label{firstrem}
Let $R^{(1)},R^{(2)}$, be two solutions to (\ref{skewbesq}) with same Brownian motion, on the 
same filtered probability space $(\Omega, \F, (\F_t)_{t\ge 0}, P)$, and such that $R^{(1)}_0=R^{(2)}_0$ $P$-a.s. Then:
\begin{itemize}
\item[(i)] The following representation formula holds for $\ell_s=\ell^{0+}_s$, 
$\ell_s=\ell^{0-}_s$, and $\ell_s=\ell^0_s$:\\
$$
\ell_t(R^{(1)}\vee R^{(2)}-\lambda^2)=
\int_0^t\mathbb{I}_{\{R^{(2)}_s-\lambda^2(s)<0\}} d\ell_s(R^{(1)}-\lambda^2)+
\int_0^t\mathbb{I}_{\{R^{(1)}_s-\lambda^2(s)\le 0\}}d\ell_s(R^{(2)}-\lambda^2).
$$
\item[(ii)]  The supremum $R^1\vee R^2$, and the infimum $R^1\wedge R^2$, are also solutions to 
(\ref{skewbesq}).
\item[(iii)]  For the supremum $S:=R^1\vee R^2$, and the infimum $I:=R^1\wedge R^2$, it holds $P$-a.s. that 
$$
S_{t}\mathbb{I}_{\Omega\setminus\{S_t>0\}\cap\{I_t\ge 0\}}=
I_{t}\mathbb{I}_{\Omega\setminus\{S_t>0\}\cap\{I_t\ge 0\}} \ \ \ \forall t\ge 0.
$$
\end{itemize}
\end{lem}
\begin{proof}
(i) Since $R^{(1)},R^{(2)}$, are continuous semimartingales w.r.t. $P$, the same is true for 
$R^{(1)}-\lambda^2,R^{(2)}-\lambda^2$. Using that 
$\ell_s((R^{(i)}-\lambda^2)-(R^{(j)}-\lambda^2))=\ell_s(R^{(i)}-R^{(j)})\equiv 0$, for $i\not=j$, 
$i,j\in \{1,2\}$, (i) 
follows from Lemma \ref{replocaltime}.
Note that if $\lambda^2$ is absolutely continuous, then upper, lower, and symmetric local times are constant 
multiples of each other (see Lemma \ref{discont}, and Lemma \ref{timezero}(iii)), and the statement 
would follow immediately from \cite[\mbox{Lemme, p.74}]{Wein}. \\  
(ii) Writing $R^{(1)}_t\vee R^{(2)}_t=(R^{(1)}_t-R^{(2)}_t)^+ +R^{(2)}_t$ and applying Tanaka's formula 
(cf. e.g. \cite[VI.(1.2)]{RYor}), we easily obtain after some calculations
\begin{eqnarray*}
R^{(1)}_t\vee R^{(2)}_t & = & R^{(1)}_0\vee R^{(2)}_0+\int_0^t\sigma\sqrt{|R^{(1)}_s\vee R^{(2)}_s|}dW_s+
\int_0^t\frac{\sigma^2}{4}(\delta-bR^{(1)}_s\vee R^{(2)}_s)dt\\
&&\hspace*{-2cm}+(2p-1)\left 
\{\int_0^t\mathbb{I}_{\{R^{(2)}_s-\lambda^2(s)<0\}} d\ell^0_s(R^{(1)}-\lambda^2)+
\int_0^t\mathbb{I}_{\{R^{(1)}_s-\lambda^2(s)\le 0\}}d\ell^0_s(R^{(2)}-\lambda^2)
\right \}.
\end{eqnarray*}
Now, we just use (i) and conclude that $R^1\vee R^2$ is another solution. 
Clearly, by linearity $R^1\wedge R^2$ 
is also a solution. One just has to use the formula
$$
\ell^{0}(X\vee Y)+\ell^{0}(X\wedge Y)=\ell^{0}(X)+\ell^{0}(Y),
$$
which can easily be derived from the corresponding formula for upper local times (see e.g. \cite{RYor}, \cite{Yan2}, \cite{ouk}).\\
(iii)  Define the function
\[ h(x):= \left\{ \begin{array}{r@{\quad\quad}l}
 -\int_0^{-x}y^{\frac{\delta}{2}}e^{\frac{by}{2}} dy& \mbox{ for}\ x<0; \\ 
  0 \ \ \ \ \ \ \ & \mbox{ for}\ x\ge 0. \end{array} \right. \] \\ \\
$h$ is continuously differentiable with locally integrable second derivative. We may hence apply It\^o's
formula with $h$. Note that $h$ is a harmonic function, i.e. $h(R)$ is a local martingale for any solution $R$ of (\ref{skewbesq}). Further, $h$ is strictly increasing in $(-\infty, 0]$. After taking expectations and stopping 
w.r.t. $\tau_n:=\inf\{t\ge 0||S_t|\ge n\}$ we obtain
\begin{eqnarray*}
E[h(S_{t\wedge \tau_n})-h(I_{t\wedge \tau_n})]=0\\
\end{eqnarray*}
for any $t\ge 0$. Letting $n\to\infty$ we get $h(S_{t})=h(I_{t})$ $P$-a.s. By continuity of the sample paths, this holds simultaneously for all $t\ge0$.
Decomposing $\Omega$ in disjoint sets
$$
\{S_t>0\}\cap\{I_t\ge 0\}, \ \{S_t>0\}\cap\{I_t< 0\}, \ \{S_t\le 0\},
$$
we get
$$
h(S_{t})\mathbb{I}_{\{S_t\le0\}}=h(I_{t})\mathbb{I}_{\{S_t\le 0\}}+h(I_{t})
\mathbb{I}_{\{S_t>0\}\cap\{I_t<0\}},
$$
and then
$$
\mathbb{I}_{\{S_t>0\}\cap\{I_t<0\}}=0,
$$
as well as  
$$
S_{t}\mathbb{I}_{\{S_t\le0\}}=I_{t}\mathbb{I}_{\{S_t\le0\}} 
$$
immediately follow.
\end{proof}\\ 
\begin{rem}\label{first}
(i) In order to obtain pathwise uniqueness for (\ref{skewbesq}) is is enough to show that the expectation $E[R_t]$ 
is uniquely determined by (\ref{skewbesq}) for all $t\ge 0$.  Indeed, if the latter holds then $E[S_t-I_t]=0$ and the 
result follows (cf. Lemma \ref{firstrem}(ii), and (iii) for the definition of $S$ and $I$). 
Unfortunately, it turns out that the determination of $E[R_t]$ seems to be rather difficult. Therefore we proceed as indicated in Remark \ref{procedure}.  \\
(ii) Suppose that we replace $(2p-1)\ell^0(R-\lambda^2)$ by $\frac{2p-1}{2p}\ell^{0+}(R-\lambda^2)$ 
(resp. $\frac{2p-1}{2(1-p)}\ell^{0-}(R-\lambda^2)$) in (\ref{skewbesq}). 
Using Lemma \ref{replocaltime} one can see as in the proof of 
Lemma \ref{firstrem},
that with 
any two solutions to the modified equation (\ref{skewbesq}) the sup and inf is again a solution. 
Consequently, as will be seen below, pathwise uniqueness can also be derived for the modified 
equation under the same assumptions. However, we work with symmetric local times, 
because the (Revuz) measures associated to symmetric local times appear naturally in integration 
by parts formulas for the corresponding Markov process generators w.r.t. some invariant or subinvariant measure (see e.g. \cite[p. 391, and section 2.1.3]{Tr5}, and \cite[section 3.1.(a)]{RuTr4} for the skew Brownian motion). \\
\end{rem}
\begin{lem}\label{replocal2}  
Let $X$ be a continuous semimartingale. 
Let $f$ be a strictly increasing function on $\R$, which is the difference of two convex functions. 
\begin{itemize}
\item[(i)] We have a.s. for any $a\in \R$
$$
\ell^{f(a)\pm}_t(f(X))=f'^{\pm}(a)\ell^{a\pm}_t(X);\ \ t\ge 0.
$$
In particular, if $R$ is a solution to (\ref{skewbesq}), then $P$-a.s.
$$
\ell^0_t(f(R-\lambda^2)-f(0))=\frac{f'^{+}(0)}{2}\ell^{0+}_t(R-\lambda^2)+
\frac{f'^{-}(0)}{2}\ell^{0-}_t(R-\lambda^2);\ \ t\ge 0,
$$
where $f'^{-}$ denotes the left hand derivative (resp. $f'^{+}$ the right hand derivative) of $f$.\\
\item[(ii)] If $f$ is additionally continuously differentiable, and 
$R$ is a solution to (\ref{skewbesq}), then $P$-a.s. 
$$
\ell_t^0(f(R)-f(\lambda^2))=\int_0^t f'(\lambda^2(s))d\ell_s^0(R-\lambda^2).
$$
\end{itemize}
\end{lem}
\begin{proof}
For (i) see  \cite[\mbox{Remark, p.222}]{oukrut}. For (ii) see \cite[\mbox{Corollary 2.11}]{oukrut}
\end{proof}\\ \\
For the purposes of this section we indicate two special It\^o-Tanaka formulas in the next lemma. 
The derivation of these formulas takes advantage of the fact that the time dependency is put into a  
semimartingale structure. Lemma \ref{itotanaka} is useful, and allows $\lambda^2$ just to be locally of bounded variation.\\ \\
For $F:\R^+\times \R \to \R$ we set 
$$
LF(t,x)=\frac{\sigma^2}{2}|x|\partial_{xx}F(t,x)+\frac{\sigma^2}{4}(\delta-bx)\partial_xF(t,x),
$$
whenever this makes sense. In what follows we shall use the notations $f^{(\prime)}, f^{(\prime\prime)}$, for 
distributional derivatives in general. \\
\begin{lem}\label{itotanaka}
Let $f$ be a strictly increasing function on $\R$, which is the difference 
of two convex functions. Assume (for simplicity) that  
$f^{(\prime\prime)}$ is locally integrable. Let 
$$
\overline{g}(y):=\gamma \mathbb{I}_{\{y< 0\}}+\frac{\alpha+\gamma}{2}\mathbb{I}_{\{y=0\}}+\alpha\mathbb{I}_{\{y>0\}};\ \ \  \alpha,\gamma, y\in \R.\\ \\
$$
\begin{itemize}
\item[(i)] Put $F(t,x)=f(x-\lambda^2(t))-f(0)$ and  
$$
H(t,x):=\overline{g}(x-\lambda^2(t))F(t,x).
$$
Then $P$-a.s.
\begin{eqnarray*}
H(t,R_t)&=&H(0,R_0) +\int_0^t\overline{g}(R_s-\lambda^2(s))f^{(\prime)}(R_s-\lambda^2(s))\sigma\sqrt{|R_s|}dW_s\\
&&\hspace*{-1.5cm}+\int_0^t\overline{g}(R_s-\lambda^2(s))\left \{LF(s,R_s)ds-f^{(\prime)}(R_s-\lambda^2(s))d\lambda^2(s)\right \}\\
&&+(\alpha p-\gamma(1-p))f'(0)\ell_t^0(R-\lambda^2).
\end{eqnarray*}
\item[(ii)] Put $F(t,x)=f(x)-f(\lambda^2(t))$ and 
$$
H(t,x):=\overline{g}(x-\lambda^2(t))F(t,x).
$$
Then $P$-a.s.
\begin{eqnarray*}
H(t,R_t)&=&H(0,R_0)  +\int_0^t\overline{g}(R_s-\lambda^2(s))f^{\prime}(R_s)\sigma\sqrt{|R_s|}dW_s\\
&&\hspace*{-1cm}+\int_0^t\overline{g}(R_s-\lambda^2(s))\left \{LF(s,R_s)ds-f^{\prime}(\lambda^2(s))d\lambda^2(s)\right \}\\
&&+(\alpha p-\gamma(1-p))\int_0^t f'(\lambda^2(s))d\ell_s^0(R-\lambda^2).
\end{eqnarray*}
\end{itemize}
\end{lem}
\begin{proof}
Since $f$ is the difference of convex functions, we know that $f\in H^{1,1}_{loc}(\R)$. Since $f^{(\prime\prime)}\in L^1_{loc}(\R)$ we obtain $f^{(\prime)}\in H^{1,1}_{loc}(\R)$. Thus in particular 
$f\in C^1(\R)$ and $f^{(\prime\prime)}\in L^1_{loc}(\R)$.\\ \\
(i) Applying the symmetric It\^o-Tanaka formula 
(cf. \cite[\mbox{VI. (1.5) Theorem}]{RYor} for the right (or upper) version), we obtain
\begin{eqnarray*}
H(t,R_t)&=&\alpha(f(R_t-\lambda^2(t))-f(0))^+ -\gamma(f(R_t-\lambda^2(t))-f(0))^-\\
&=&H(0,R_0)+\int_0^t\overline{g}(R_s-\lambda^2(s))df(R_s-\lambda^2(s))
+\frac{\alpha-\gamma}{2}\ell_t^0(f(R-\lambda^2)-f(0)).\\
\end{eqnarray*}
Applying again the symmetric It\^o-Tanaka formula, (\ref{occtim}), and Lemma \ref{replocal2}(i), the right hand side equals 
\begin{eqnarray*}
&&H(0,R_0)+\int_0^t\overline{g}(R_s-\lambda^2(s))f^{(\prime)}(R_s-\lambda^2(s))d(R_s-\lambda^2(s))\\
&&+\int_0^t\overline{g}(R_s-\lambda^2(s))\frac{\sigma^2}{2} |R_s|f^{(\prime\prime)}(R_s-\lambda^2(s))ds
+\frac{\alpha-\gamma}{2}f'(0)\ell_t^0(R-\lambda^2),\\
\end{eqnarray*}
which easily leads to the desired conclusion.\\
(ii) Using Lemma \ref{replocal2}(ii) instead of Lemma \ref{replocal2}(i) the proof of (ii) is nearly the same than the proof of (i). We therefore omit it.
\end{proof}\\ \\
\begin{rem}\label{procedure}
If $\alpha,\gamma$, are strictly positive, then 
$$
H(t,x):=\overline{g}(x-\lambda^2(t))(f(x-\lambda^2(t))-f(0)),
$$
or
$$
H(t,x):=\overline{g}(x-\lambda^2(t))(f(x)-f(\lambda^2(t))),
$$
is strictly increasing in $x$, whenever $f$ is. Moreover functions of this type 
allow to get rid of the local time $\ell^0(R-\lambda^2)$. Below, we will apply Gronwall's 
inequality (see Theorem \ref{Gronwall}) to functions 
$$ 
g(t)=E[H(t,S_t)-H(t,I_t)],
$$
using the It\^o-Tanaka  formula of Lemma \ref{itotanaka} (resp. apply Peskir's It\^o-Tanaka formula in 
Theorem \ref{Ito}), and derive pathwise uniqueness in Theorem \ref{pathunique2} (resp. \ref{pathunique}). For this purpose it is important to 
find nice functions $f$ (see Theorem \ref{pathunique2}, \ref{pathunique}). 
\end{rem}
As an application of the preceding Lemma \ref{itotanaka}, we present the next corollary. 
It provides for some special $\lambda$'s a different proof of the fact that is derived in Remark \ref{nosolution+}(ii) for general time dependent $\lambda$. 
The idea for its proof is similar to the idea used in \cite{hs} to show that the $\alpha$-skew Brownian motion doesn't exist if $|\alpha|>1$ ($\alpha=2p-1$). \\
\begin{cor}\label{nosolution}
(i) Let $R_0=\lambda^2(0)$, and $d\lambda^2(t)=\frac{\sigma^2}{4}\left \{\delta-b \lambda^2(t)\right \}dt$, 
or $d\lambda^2(t)=\frac{\sigma^2 \delta}{4} dt$. 
Then there is no solution to (\ref{skewbesq}), if $|2p-1|>1$.\\
(ii) Let $0<R_0=c\equiv \lambda^2$, Then there is no solution to (\ref{skewbesq}), if $|2p-1|>1$.
\end{cor}
\begin{proof}
(i) Let us to the contrary assume that there is a solution. Then we can apply Lemma \ref{itotanaka}(i) 
with $f(x)=x$, and $\alpha=p-1$, $\gamma=-p$, if $p>1$ 
(resp. $\alpha=1-p$, $\gamma=p$, if $p<0$). 
If $d\lambda^2(t)=\frac{\sigma^2}{4}\left \{\delta-b \lambda^2(t)\right \}dt$, it follows 
$$
0\le H(t,R_t)\le \int_0^t\overline{g}(R_s-\lambda^2(s))\sigma\sqrt{R_s}dW_s, \ \ 0\le t<\infty,\\
$$
which holds pathwise, hence also with $t$ replaced by $t\wedge \tau_n$, where 
$\tau_n:=\inf\{t\ge 0||R_t|\ge n\}$. Clearly $\tau_n\nearrow \infty$ $P$-a.s. It follows that the 
$P$-expectation of $H(t,R_t)$ is zero, hence $R\equiv\lambda^2$  $P$-a.s., which is impossible. 
In case  $d\lambda^2(t)=\frac{\sigma^2 \delta}{4} dt$ we 
first note that $R_0=\lambda^2(0)\ge 0$, implies $P$-a.s. 
$\frac{\sigma^2}{2}|R_t|=\frac{\sigma^2}{2}R_t$ for all $t$, by Lemma \ref{timezero}(ii). 
Then we apply Lemma \ref{itotanaka}(i) with $f(x)=e^{\frac{bx}{2}}$ and conclude in the same manner as before 
with $f(x)=x$.\\
(ii) Let us to the contrary assume that there is a solution. Let $g:\R\to\R^+$ be such that
$g(x)=0$, if $x\le 0$, 
$g\in C^1(\R)$, $g(x)=x^{\frac{\delta}{2}+2}$, for $x\in [0,\frac{c}{2}]$.
Suppose further that $g'(x)$ is negative if $x\ge c$, and positive if $x\le c$. Define 
\[ f_g(x):= \left\{ \begin{array}{r@{\quad\quad}l}
 -\int_0^{-x}y^{\frac{\delta}{2}}e^{\frac{by}{2}} dy\ \ & \mbox{ for}\ x<0; \\ 
  \int_0^x g(y)y^{-\frac{\delta}{2}}e^{\frac{by}{2}}dy&\mbox{ for}\ x\ge 0. \end{array} \right. \] \\ 
Then $f_g\in C^1(\R)$ is strictly increasing, with locally integrable second derivative, and 
$$
Lf_g(x)=\frac{\sigma^2}{2}x^{1-\frac{\delta}{2}}e^{\frac{bx}{2}}g'(x)
\mathbb{I}_{x\not=0}.
$$
Now, we can apply Lemma \ref{itotanaka}(ii) 
with $f(x)=f_g(x)$
, and $\alpha=p-1$, $\gamma=-p$, if $p>1$ 
(resp. $\alpha=1-p$, $\gamma=p$, if $p<0$). It follows 
$$
0\le H(t,R_t)= \int_0^t\overline{g}(R_s-c)f'_g(R_s)\sigma\sqrt{R_s}dW_s
+\frac{\sigma^2}{2}\int_0^t \overline{g}(R_s-c)R_s^{1-\frac{\delta}{2}}e^{\frac{bR_s}{2}}g'(R_s)ds.
$$
By our assumptions on $g$, the bounded variation part is non-positive. Thus we may conclude analogously 
to (i), that
$f_g(R_t)=f_g(c)$, and hence $R\equiv c$, which is impossible.
\end{proof}\\ \\
We will make use of the following generalization of Gronwall's inequality.
Its proof can be found in \cite[\mbox{Appendixes, 5.1. Theorem}]{ek}.
\begin{theo}\label{Gronwall}
Let $\mu^+$ be a Borel measure (finite on compacts!) on $[0,\infty)$, let $\varepsilon\ge 0$, and let $g$ be a Borel measurable function that is bounded on bounded intervals and satisfies
$$
0\le g(t)\le \varepsilon+\int_{[0,t)}g(s)\mu^+(ds), \ \ \ t\ge 0.
$$
Then 
$$
g(t)\le \varepsilon e^{\mu^+([0,t))}, \ \ \ t\ge 0.
$$
\end{theo}
We are now prepared to formulate our main theorem.\\
\begin{theo}\label{pathunique2}
Let $f$ be 
a strictly increasing function on $\R$, which is the difference of two convex functions. 
Let $f^{(\prime\prime)}$ be locally integrable.
Let either  $F(t,x)=f(x-\lambda^2(t))-f(0)$  or $F(t,x)=f(x)-f(\lambda^2(t))$. Suppose further that 
\begin{eqnarray}\label{out}
(\partial_t+L)F(t,x)= F(t,x)\mu(dt)+sgn(2p-1)\nu(dt) \ \mbox{ for (a.e.) } x\ge 0,
\end{eqnarray}
where $\mu(dt)=\mu^+(dt)-\mu^-(dt)$ is a signed Borel measure, with continuous positive part $\mu^+(dt)$, $\nu(dt)$ is a positive Borel measure,
(\ref{out}) is in the sense of distributions, and $sgn$ is the point-symmetric sign function, i.e. $sgn(x)=-1$, if $x<0$, $sgn(0)=0$, and $sgn(x)=1$, if $x>0$. Then pathwise uniqueness holds for (\ref{skewbesq}).
\end{theo}
\begin{proof}
Let $\overline{g}$ be as in Lemma \ref{itotanaka}, 
with $\alpha=1-p$, $\gamma=p$, and
$$
H(t,x):=\overline{g}(x-\lambda^2(t))F(t,x),
$$
Let $R^{(1)},R^{(2)}$, be two solutions to (\ref{skewbesq}) 
with same Brownian motion, same initial condition, and on the 
same filtered probability space $(\Omega, \F, P)$. By Lemma \ref{firstrem} 
we know that $S=R^{(1)}\vee R^{(2)}$, and $I=R^{(1)}\wedge R^{(2)}$, are also solutions to (\ref{skewbesq}). 
Define the stopping time $\tau_n:=\inf\{t\ge 0:|S_t| \ge n\}$.  Then clearly $\tau_n\nearrow \infty$ $P$-a.s. 
Applying  Lemma \ref{itotanaka}, we obtain for $Z=S$, and for $Z=I$, 
\begin{eqnarray*}
E[H(t\wedge\tau_n,Z_{t\wedge\tau_n})]&=&E[H(0,Z_0)]\\
&&\hspace*{-2cm}+E\left [\int_0^{t\wedge\tau_n}\overline{g}(Z_s-\lambda^2(s))
d\left \{\int_0^s LF(u,Z_u)du-\int_0^s f^{(\prime)}(\widetilde{Z}_u)d\lambda^2(u)\right \}\right ],
\end{eqnarray*}
where either $\widetilde{Z}_s=Z_s-\lambda^2(s)$ (in case $F(t,x)=f(x-\lambda^2(t))-f(0)$), or $\widetilde{Z}_s=\lambda^2(s)$ (if $F(t,x)=f(x)-f(\lambda^2(t))$). 
By Lemma \ref{firstrem}(iii) we know that $P$-a.s.
$$
S_{t}\mathbb{I}_{\Omega\setminus\{S_t>0\}\cap\{I_t\ge0\}}=
I_{t}\mathbb{I}_{\Omega\setminus\{S_t>0\}\cap\{I_t\ge0\}} \ \ \ \forall t\ge 0.
$$
We can therefore neglect what happens outside $\{S_t> 0\}\cap\{I_t\ge0\}$. Thus, by assumption (\ref{out})
$$
\hspace*{-4cm}E\left [H(t\wedge\tau_n,S_{t\wedge\tau_n}) -H(t\wedge\tau_n,I_{t\wedge\tau_n})\right ]
$$
$$
= E\left [\int_0^{t\wedge\tau_n} \left( H(s,S_s)-H(s,I_s)\right )\mu(ds)\right ]
$$
$$
+sgn(2p-1)E\left [\int_0^{t\wedge\tau_n} \left( \overline{g}(S_s-\lambda^2(s))-
\overline{g}(I_s-\lambda^2(s))\right )\nu(ds)\right ], 
$$
which is further, since $sgn(2p-1)\overline{g}$ is decreasing, estimated from above by 
\begin{eqnarray*}
E\left [\int_0^{t\wedge\tau_n} \left( H(s,S_{s})-H(s,I_{s})\right )
\mu^+(ds)\right ],
\end{eqnarray*}
and then again, since $H(s,S_{s})-H(s,I_{s})$ is positive, by 
\begin{eqnarray*}
E\left [\int_0^t \left( H(s\wedge\tau_n,S_{s\wedge\tau_n})-
H(s\wedge\tau_n,I_{s\wedge\tau_n})\right )\mu^+(ds)\right ].
\end{eqnarray*}
Applying Fubini's theorem and Theorem \ref{Gronwall}, we obtain that 
$$
E\left [H(t\wedge\tau_n,S_{t\wedge\tau_n}) -H(t\wedge\tau_n,I_{t\wedge\tau_n})\right ]=0, \ \ 0\le t <\infty.
$$
Since $H$ increases in the space variable, for any fixed time, it follows that 
$S_{\cdot\wedge\tau_n}$ and $I_{\cdot\wedge\tau_n}$ are $P$-indistinguishable. 
Letting $n\to \infty$ we see that $S =I$, hence $R^{(1)}= R^{(2)}$, and pathwise uniqueness is shown.
\end{proof}

\begin{cor}\label{pathunique0}
Let $\overline{p}:=sgn(2p-1)$. Pathwise uniqueness holds for 
(\ref{skewbesq}), whenever 
$$
\overline{p}d\lambda^2(s)\le \overline{p}\frac{\sigma^2}{4}\left \{\delta-
\left (\frac{1-\overline{p}}{2}\right )b \lambda^2(s)\right \}ds.
$$ 
\end{cor} 
\begin{proof}
Let $f(x)=x$, and $F(t,x)=f(x-\lambda^2(t))-f(0)$. Then 
$$
(\partial_t+L)F(t,x)= -\frac{\sigma^2b}{4}F(t,x)dt+\frac{\sigma^2}{4}\left (\delta-b\lambda^2(t)\right )dt-d\lambda^2(t)
$$
Putting $\mu(dt)=-\frac{\sigma^2b}{4}dt$, and 
$$
\nu(dt)=sgn(2p-1)\left \{\frac{\sigma^2}{4}\left (\delta-b\lambda^2(t)\right )dt-d\lambda^2(t)\right \},
$$
we conclude by Theorem \ref{pathunique2} that pathwise uniqueness holds, if 
$$
sgn(2p-1)d\lambda^2(s)\le sgn(2p-1)\frac{\sigma^2}{4}\left \{\delta-b \lambda^2(s)\right \}ds.
$$
This holds for $p\in (0,1)$, and $b\ge 0$. If $2p-1>0$, and $b>0$, we may refine our argument letting $f(x)=e^{\frac{bx}{2}}$. Then
$$
(\partial_t+L)F(t,x)= \frac{b}{2}F(t,x)d\left \{\frac{\sigma^2 \delta}{4}t-\lambda^2(t)\right \}+
\frac{b}{2}\left \{\frac{\sigma^2 \delta}{4}dt-d\lambda^2(t)\right \}\ \mbox{ for a.e. } x\ge 0,
$$
and we apply again  Theorem \ref{pathunique2} with $\mu^+(dt)=\nu(dt)=\frac{b}{2}\left \{\frac{\sigma^2 \delta}{4}dt-d\lambda^2(t)\right \}$, so that pathwise uniqueness holds whenever
$$
d\lambda^2(t)\le \frac{\sigma^2 \delta}{4}dt.
$$
Combining both cases, we obtain the statement.
\end{proof}

\section{The martingale problem and pathwise uniqueness via Peskir's It\^o-Tanaka formula}\label{three}
In this section we shall provide two additional tools to potentially improve existence results and the pathwise uniqueness results of the preceding section. 
The first one (see Proposition \ref{mart}) 
is related to the martingale problem for $R$ and is interesting for its own (see also Remark \ref{alg} for its usefulness).
The second tool is a complement to Theorem \ref{pathunique2} in the sense that Theorem \ref{pathunique} can be applied to \lq\lq true\rq\rq\ time dependent functions.
The general criterion for pathwise uniqueness in Theorem \ref{pathunique2} uses special time dependent functions $F$ built on functions $f$ 
that do not depend on time.  
We will use Peskir's It\^o-Tanaka formula (see $\cite[\mbox{Theorem 2.1}]{peskir}$).
\\ \\ 
Let 
$$
\Gamma(\lambda^2):=\{(s,x)\in \R^+\times\R^+| x=\lambda^2(s)\},
$$
and
$$
C:=\{(s,x)\in \R^+\times\R^+| x<\lambda^2(s)\}, \ \ D:=\{(s,x)\in \R^+\times\R^+| x>\lambda^2(s)\}.
$$\\
For $F:\R^+\times \R\to \R$ we consider the linear operator
\begin{eqnarray*}
{\cal L}F(t,x) & = & \frac{\sigma^2}{2}|x|\partial_{xx}F(t,x)+\frac{\sigma^2}{4}(\delta-bx)\partial_xF(t,x)+
\partial_tF(t,x)\\
& = & (\partial_t+L)F(t,x)
\end{eqnarray*}
whenever this expression makes sense.\\
Let further\\
\begin{eqnarray*}
{\cal M}&:=&C(\R^+\times \R)\cap\{F\in C^{1,2}(\R^+\times \R\setminus \Gamma(\lambda^2))\,|\,
F \mbox{ is } C^{1,2} \mbox{ on } \overline{C} \mbox{ and on }\overline{D}\}.\\
\end{eqnarray*} 
${\cal M}$ is exactly the space of functions for which the It\^o-Tanaka formula is derived in \cite{peskir} (cf. \cite[page 500]{peskir}). 
\begin{prop}\label{mart}
Define
\begin{eqnarray*}
D({\cal L})&:=&C_0(\R^+\times \R)\cap \{F\in {\cal M}\,|\,  (1-p)\partial_xF(t,\lambda^2(t)-)=p\partial_xF(t,\lambda^2(t)+) \}.
\end{eqnarray*} 
Let $F\in D({\cal L})$, and $R$ be a weak solution to (\ref{skewbesq}). Then 
$$
\left (F(t,R_t)-F(0,R_0)-\int_0^t {\cal L}F(s,R_s)ds\right )_{t\ge 0}
$$
is a $P$-martingale.\\
\end{prop}
\begin{proof}
First observe that $\int_0^t {\cal L}F(s,R_s)ds$, $F\in D({\cal L})$, is well-defined by 
Lemma \ref{timezero}(iv). 
By \cite[\mbox{Theorem 2.1}]{peskir} and Lemma \ref{timezero}(iv), we obtain $P$-a.s.
\begin{eqnarray}\label{itos}
F(t,R_t)&=&F(t, R_0)+\int_0^t \partial_tF(s,R_s)ds \nonumber \\
&&+\int_0^t\partial_xF(s,R_s)\sigma\sqrt{|R_s|}dW_s\nonumber \\
&&+\int_0^t\frac{\sigma^2}{4}(\delta-bR_s)\partial_xF(s,R_s)ds\nonumber \\
&&+\int_0^t\frac12(\partial_xF(s,\lambda^2(s)+)+\partial_xF(s,\lambda^2(s)-))(2p-1)d\ell_s^0(R-\lambda^2)\nonumber \\
&&+\int_0^t\frac{\sigma^2}{2}|R_s|\partial_{xx}F(s,R_s)ds\nonumber \\
&&+\frac12\int_0^t(\partial_xF(s,R_s+)-\partial_xF(s,R_s-))d\ell_s^0(R-\lambda^2).
\end{eqnarray}
Since $F\in D({\cal L})$ we have 
$$
\partial_xF(s,\lambda^2(s)+)-\partial_xF(s,\lambda^2(s)-)=\frac{1-2p}{(1-p)}\partial_xF(s,\lambda^2(s)+),
$$
and 
$$
\partial_xF(s,\lambda^2(s)+)+\partial_xF(s,\lambda^2(s)-)=\frac{1}{(1-p)}\partial_xF(s,\lambda^2(s)+),
$$
so that the expressions with $\ell^0_s(R-\lambda^2)$ in (\ref{itos}) cancel each other. Therefore
$$
F(t,R_t)-F(0,R_0)-\int_0^t {\cal L}F(s,R_s)ds=\int_0^t\partial_xF(s,R_s)\sigma\sqrt{|R_s|}dW_s.
$$
By our further assumptions on $F$ the left hand side is square integrable and thus the result follows.
\end{proof}
\begin{rem}\label{alg}
$D({\cal L})$ is an algebra of functions that separates (at least for $\lambda^2\in C^1(\R^+)$) the points of $\R^+\times \R$, thus well suited as a starting point to study existence and uniqueness 
in law of $R$ (cf. e.g. \cite[\mbox{chapter 4}]{ek}). In particular, if only the expectation of $R_t$ is uniquely determined through (\ref{skewbesq}) for any fixed $t\ge 0$, 
then pathwise uniqueness follows for (\ref{skewbesq}) (see Remark \ref{first}(i)).
\end{rem}
By Lemma \ref{timezero}(iv) 
$$
\left (\int_0^t G(s,R_s){\cal{L}}H(s,R_s)ds\right )_{t\ge 0}\\
$$\\
is well-defined for any $H\in {\cal M}$, and $G$ bounded and measurable.
\begin{lem}\label{Ito}
Let $F\in C^{1,2}(\R^+\times\R)$, such that $F(t,\lambda^2(t))=0$ for all $t\ge 0$. Set $H(t,x)=\overline{g}(x-\lambda^2(s))F(t,x)$, 
where $\overline{g}$ is defined as in Lemma 
\ref{itotanaka}. Then 
$H\in {\cal M}$, and 
\begin{eqnarray*}
&&H(t,R_t) = H(0,R_0)+\int_0^t\overline{g}(R_s-\lambda^2(s))\partial_xF(s,R_s)\sigma\sqrt{|R_s|}dW_s\\
&&+\int_0^t \overline{g}(R_s-\lambda^2(s)){\cal{L}}F(s,R_s)ds
+\left (\alpha p-\gamma(1-p)\right )
\int_0^t\partial_xF(s,\lambda^2(s))d\ell^0_s(R-\lambda^2).
\end{eqnarray*}
\end{lem}
\begin{proof}
The first statement is clear. 
Applying $\cite[\mbox{Theorem 2.1}]{peskir}$, Lemma \ref{timezero}(iv), and noting that 
$$
\partial_xH(s,\lambda^2(s)+)-\partial_xH(s,\lambda^2(s)-)=(\alpha-\gamma)\partial_xF(s,\lambda^2(s)),
$$
and 
$$
\partial_xH(s,\lambda^2(s)+)+\partial_xH(s,\lambda^2(s)-)=(\alpha+\gamma)\partial_xF(s,\lambda^2(s)),
$$\\
similarly to the proof of Proposition \ref{mart} we get the result.
\end{proof}\\

\begin{theo}\label{pathunique}
Let $\beta(t)\in L^1_{loc}(\R)$. Let $F\in C^{1,2}(\R^+\times \R)$ be such that $F(t,x)$ is strictly increasing in $x$ for every fixed $t\ge 0$, and 
$$
F(t,\lambda^2(t))=0 \ \ \ \forall t\ge 0.
$$
Let $H(t,x):=\overline{g}(x-\lambda^2(s))F(t,x)$,  
where $\overline{g}$ is as in Lemma \ref{itotanaka}, 
with $\alpha=1-p$, $\gamma=p$. Suppose further that
\begin{eqnarray}\label{out2}
{\cal{L}}H(t,x)= \beta(t) H(t,x)+\overline{g}(x-\lambda^2(t))v(t), 
\mbox{ for } (t,x)\in \R^+\times \R^+\setminus \Gamma(\lambda^2)
\end{eqnarray}
where $v\ge 0$, if $p>\frac12$, or $v\le 0$, if $p<\frac12$. 
Then pathwise uniqueness holds for (\ref{skewbesq}).
\end{theo}
\begin{proof}
The proof is exactly the same than the proof of Theorem \ref{pathunique2}. We therefore omit it.
\end{proof}
\begin{rem}
Theorem \ref{pathunique} shows that Peskir's It\^o-Tanaka formula \cite[Theroem 2.1]{peskir} can be used to obtain a pathwise uniqueness criterion 
corresponding to \rq\rq true\lq\lq\  time-dependent functions. We could also have used \cite[Theroem 2.1]{peskir} to to obtain  
the full criterion of Corollary \ref{pathunique0}. In fact, one first shows Lemma \ref{itotanaka}(i) for $f\in C^2(\R)$ accordingly to \cite[Remark 2.2.2.]{peskir} by applying \cite[Theroem 2.1]{peskir} to the semimartingale $\widetilde{X}_t=X_t-\lambda^2(t)$.
By this one obtains Theorem \ref{pathunique2} for $F(t,x)=f(x-\lambda^2(t))-f(0)$, $f\in C^2(\R)$, which is sufficient for the proof of Corollary \ref{pathunique0}. The formula from \cite[Theroem 2.1]{peskir} has been generalized in \cite[Theorem 2.1]{peskir3} to the multidimensional case, and
in \cite{Tr6} multidimensional skew reflected diffusions with singular coefficients have been constructed. Although the multidimensional case is quite different from the special case of dimension one  we think 
that the formula of \cite[Theorem 2.1]{peskir3} might be useful to obtain pathwise uniqueness results for general skew reflected diffusions in higher dimensions. In particular, it might be useful for the skew reflected diffusions of \cite{Tr6}.
\cite[Theorem 2.1]{peskir3} might also be useful to formulate the martingale problem for higher dimensional skew reflected diffusions similarly to Proposition \ref{mart}.
\end{rem}

\end{document}